\documentclass[12pt]{article}
\usepackage{amsfonts}
\usepackage{amsfonts}
\usepackage{mathrsfs}
\usepackage{graphicx}
\usepackage{caption2}
\usepackage{amsfonts}
\usepackage[dvipsnames,usenames]{color}
\usepackage{amsmath, amsthm, amsfonts, amssymb}
\usepackage[toc,page]{appendix}

\newtheorem{thm}{Theorem}[section]
\newtheorem{cor}[thm]{Corollary}
\newtheorem{lem}[thm]{Lemma}
\newtheorem{prop}[thm]{Proposition}

\newtheorem{rem}[thm]{Remark}
\newtheorem{defn}[thm]{Definition}
\newtheorem{con}[thm]{Conjecture}
\theoremstyle{definition}
\numberwithin{equation}{section}
\theoremstyle{remark} \hsize=7.5truein \vsize=8.6truein

\def\P{{\hbox{\bf P}}}
\def\E{{\hbox{\bf E}}}

\def\be{{\bf{e}}}
 at 10 true pt

\def\be#1{ \begin{equation}\label{#1} }
\def\bas{\begin{align*}}
\def\eas{\end{align*}}
\def\bi{\begin{itemize}}
\def\ei{\end{itemize}}

\def\emph#1{{\it #1}}
\def\textbf#1{{\bf #1}}

\theoremstyle{plain}
  
  \newtheorem{question}[thm]{Question}

\theoremstyle{remark}

\theoremstyle{definition}

\begin{document}
\title{Sparse random graphs: Eigenvalues and Eigenvectors }

\author{Linh V. Tran, Van H. Vu\footnote{V. Vu is supported by NSF grants DMS-0901216 and AFOSAR-FA-9550-09-1-0167. } ~and Ke Wang\\
Department of Mathematics, Rutgers, Piscataway, NJ 08854}
\date{}
\maketitle

\begin{abstract}
In this paper we prove 
the semi-circular law for the eigenvalues of regular random graph $G_{n,d}$  in the case
 $d\rightarrow \infty$, complementing a  previous result of McKay for fixed $d$. 
 We also obtain a upper bound on  the infinity norm of  eigenvectors of Erd\H{o}s-R\'enyi random graph $G(n,p)$, answering a 
  question raised by Dekel-Lee-Linial.
\end{abstract}

\section{Introduction}

\subsection{Overview} 

In this paper, we consider two models of random graphs, the Erd\H{o}s-R\'{e}nyi random graph $G(n,p)$ and the random regular graph $G_{n,d}$. 
Given a real number $p=p(n)$,$0\le p\le 1$, the Erd\H{o}s-R\'{e}nyi graph on a vertex set of size $n$ is obtained by drawing an edge between each pair of vertices, randomly and independently, with probability $p$.  
On the other hand,  $G_{n,d}$, where $d=d(n)$ denotes the degree, is a random graph chosen uniformly from the set of 
 all simple $d$-regular graphs on $n$ vertices. These are basic models in the theory of random graphs. For further information, we  refer the readers 
  to  the excellent monographs $ \cite{bollobas} ,\cite{janson}$  and  
  survey $ \cite{MRRG}$.
  

Given a  graph $G$ on $n$ vertices, the adjacency matrix $A$ of $G$ is an $n\times n$ matrix whose entry $a_{ij}$ equals one if there 
is an edge between the vertices $i$ and $j$ and zero otherwise. All diagonal entries  $a_{ii}$ are defined to be  zero. 
  The eigenvalues and eigenvectors of $A$ carry valuable information about the structure of the graph
  and have been studied by many researchers for quite some time, with both theoretical and practical motivations (see, for example, $\cite{bauer},\cite{bhamidi},\cite{feige},\cite{semerjian}$ $\cite{FK1981},\cite{fried1991},\cite{fried2003}$, $\cite{fried1993},\cite{tvrandom},\cite{erdos2009semi}$, $\cite{shi2000},\cite{pothen1990}$).

The goal of this paper is to study the eigenvalues and eigenvectors of $G(n,p)$ and $G_{n,d}$. 
We are going to consider: 

\begin{itemize}

\item The global law  for the limit of the empirical spectral distribution (ESD) of adjacency matrices of $G(n,p)$ and $G_{n,d}$.
For $p = \omega (1/n)$, it is well-known that eigenvalues of $G(n,p)$ (after a proper scaling) follows Wigner's semicircle law 
(we include a short proof in the Appendix \ref{appendix:SCL} for completeness). 
Our main new result shows that the same law  holds  for random regular graph with $d \rightarrow \infty$ with $n$. 
This complements the well known result of McKay for the case when $d$ is an absolute constant (McKay's law) and extends recent results of 
Dumitriu and Pal \cite{SRRG} (see Section \ref{section:SCL} for more discussion).



\item Bound on the infinity norm of the eigenvectors. We first prove that the infinity norm of any (unit) eigenvector $v$ of $G(n,p)$ is almost surely $o(1)$ for $p=\omega(\log n/n)$. This gives a positive answer to a question raised by Dekel, Lee and Linial 
\cite{DLL}. Furthermore, we can show that $v$  
satisfies the bound $\|v\|_{\infty} = O\left(\sqrt { \log^{2.2} g(n){\log n}/{np} }  \right) $ for $p=\omega(\log n/n)=g(n)\log n/n$, as long as the corresponding eigenvalue 
is bounded away from the (normalized) extremal values $-2$ and $2$. 

\end{itemize}  

We finish this section with some notation and conventions. 

Given an $n\times n$ symmetric matrix $M$, we denote its $n$ eigenvalues as $${\lambda}_1 {(M)} \le {\lambda}_2 {(M)} \le \ldots \le {\lambda}_n {(M)},$$ and let $u_1(M),\ldots,u_n(M) \in \mathbb{R}^n$ be an orthonormal basis of eigenvectors of $M$ with $$M u_i(M)={\lambda}_i u_i(M).$$

The empirical spectral distribution (ESD) of the matrix $M$ is a one-dimensional function $$F^{\bf M}_n(x)=\frac{1}{n} |\{ 1\le j \le n: \lambda_j(M) \le x\}|,$$

where we use $|\mathbf{I}|$ to denote the cardinality of a set $\mathbf{I}$.

Let $A_n$ be the adjacency matrix of $G(n,p)$. Thus $A_n$ is a random symmetric $n\times n$ matrix whose upper triangular entries are iid copies of a real random variable $\xi$ and diagonal entries are $0$. $\xi$ is a Bernoulli random variable that takes values $1$ with probability $p$ and $0$ with probability $1-p$. $$\mathbb{E} \xi=p, \mathbb{V}ar{\xi}=p(1-p)={\sigma}^2.$$ 
Usually it is more convenient to study the normalized matrix 
$$M_n=\frac{1}{\sigma}(A_n-pJ_n)$$
 
where $J_n$ is the $n\times n$ matrix all of whose  entries are 1. $M_n$ has entries with mean zero and variance one. The global properties of the eigenvalues of $A_n$ and $M_n$ are essentially the same (after proper scaling), thanks to the following lemma

\begin{lem} 
\emph{(Lemma 36, \cite{tvrandom})}
\label{EigenDiff}
Let $A,B$ be symmetric matrices of the same size where $B$ has rank one. Then for any interval $I$,
$$|N_I(A+B)-N_I(A)| \le 1,$$
where $N_I(M)$ is the number of eigenvalues of $M$ in $I$.
\end{lem}

\begin{defn}
Let $E$ be an event depending on $n$. Then $E$ holds with {\it  overwhelming probability}  if $\P(E)\geq 1- \exp(-\omega(\log n))$.
\end{defn}

The main advantage of this definition is that if we have a polynomial number of events, each of which holds with overwhelming probability, then their intersection also holds with overwhelming probability. 

Asymptotic notation is used under the assumption that $n \rightarrow \infty$. For functions $f$ and $g$ of parameter $n$, we use the following notation as $n\rightarrow \infty$: $f=O(g)$ if $|f|/|g|$ is bounded from above; $f=o(g)$ if $f/g \rightarrow 0$; $f=\omega(g)$ if $|f|/|g|\rightarrow \infty$, or equivalently, $g=o(f)$; $f=\Omega(g)$ if $g=O(f)$; $f=\Theta(g)$ if $f=O(g)$ and $g=O(f)$.\\

\subsection{The semicircle law} \label{section:SCL}

In 1950s, Wigner \cite{wigner} discovered the famous semi-circle for the limiting distribution of the eigenvalues of random matrices. 
His proof extends, without difficulty, to the adjacency matrix  of $G(n,p)$, given that $np \rightarrow \infty$ with $n$. (See Figure \ref{fig:SCL1} for a numerical simulation)

\begin{thm}
\label{thm:SCL1}
For $p=\omega(\frac{1}{n})$, the empirical spectral distribution (ESD) of the matrix $\frac{1}{\sqrt{n}\sigma} A_n$ converges in distribution to the semicircle distribution  which has a density ${{}\rho}_{sc}(x)$ with support on $[-2,2]$,  $${{\rho}}_{sc}(x) := \frac{1}{2 \pi } \sqrt{4 -x^2}.$$
\end{thm}

\begin{figure}[htbp]
  \centering 
  \includegraphics[scale=0.7]{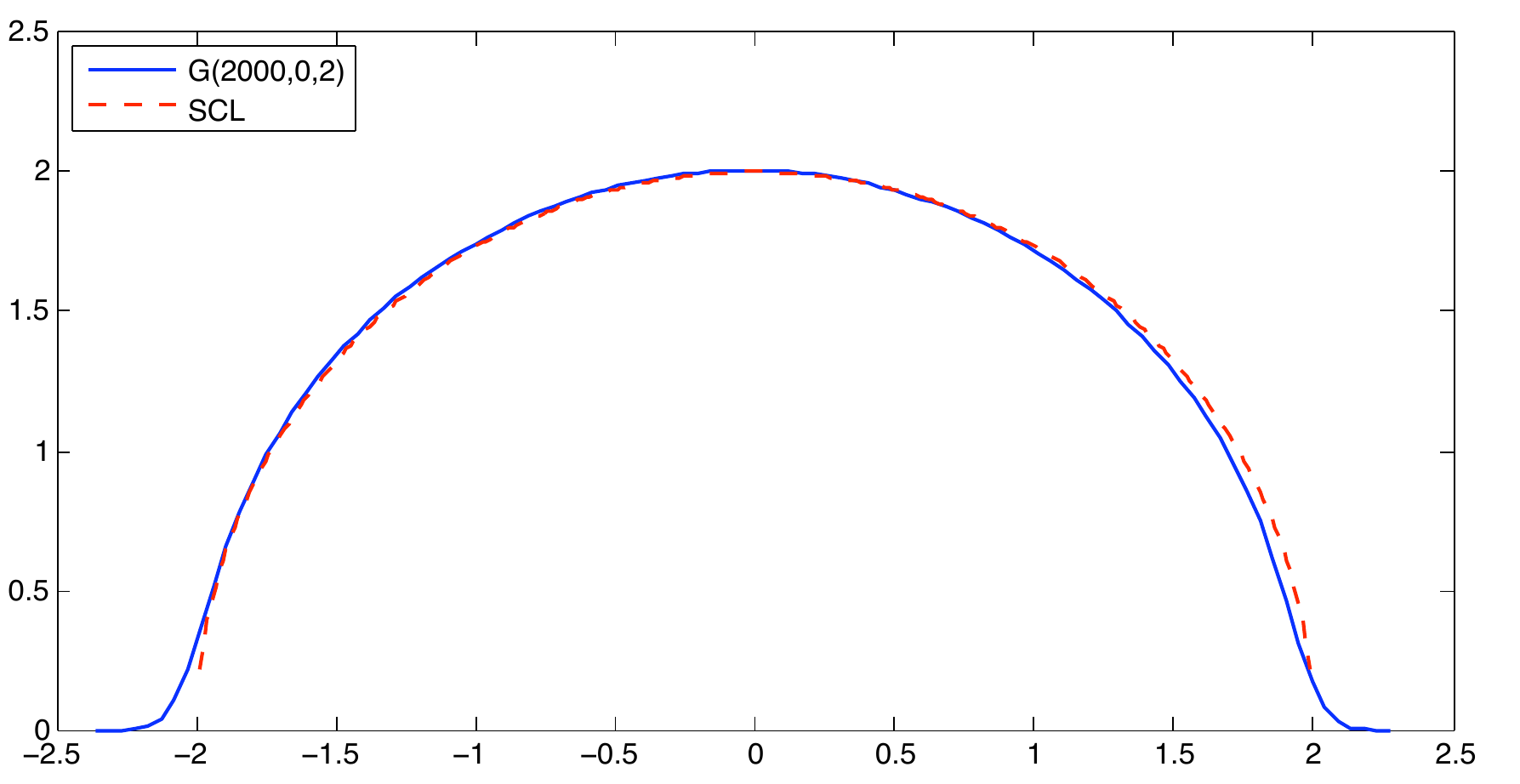}
  \captionstyle{center} 
  \onelinecaptionsfalse 
  \caption{The probability density function of the ESD of
           $G(2000,0.2)$ } 
\label{fig:SCL1}
 \end{figure}

If $np= O(1)$, the semicircle law no longer holds. In this case, the graph almost surely has $\Theta(n)$ isolated vertices, so 
in the limiting distribution, the point $0$ will have positive constant mass.

The case of  random regular graph, $G_{n,d}$, was 
  considered by McKay \cite{mckay} about 30 years ago. He proved that if $d$ is fixed, and $n \rightarrow \infty$, then the 
limiting density function is 

$$\displaystyle f_d(x)=\left\{ \begin{array}{ll}
         \frac{d\sqrt{4(d-1)-x^2}}{2\pi (d^2-x^2)}, & \mbox{if $|x| \leq 2\sqrt{d-1}$};\\
         \\
        0 & \mbox{otherwise}.\end{array} \right.$$ 

This is usually referred to as McKay or Kesten-McKay law. 

It is easy to verify that as $d\rightarrow \infty$, if we normalize the variable $x$ by $\sqrt{d-1}$, then the above density converges to the semicircle distribution on $[-2,2]$. 
In fact, a numerical simulation shows the convergence is quite fast(see Figure \ref{fig:SCL-RRG}).

\begin{figure}[htbp]
  \centering 
  \includegraphics[scale=0.6]{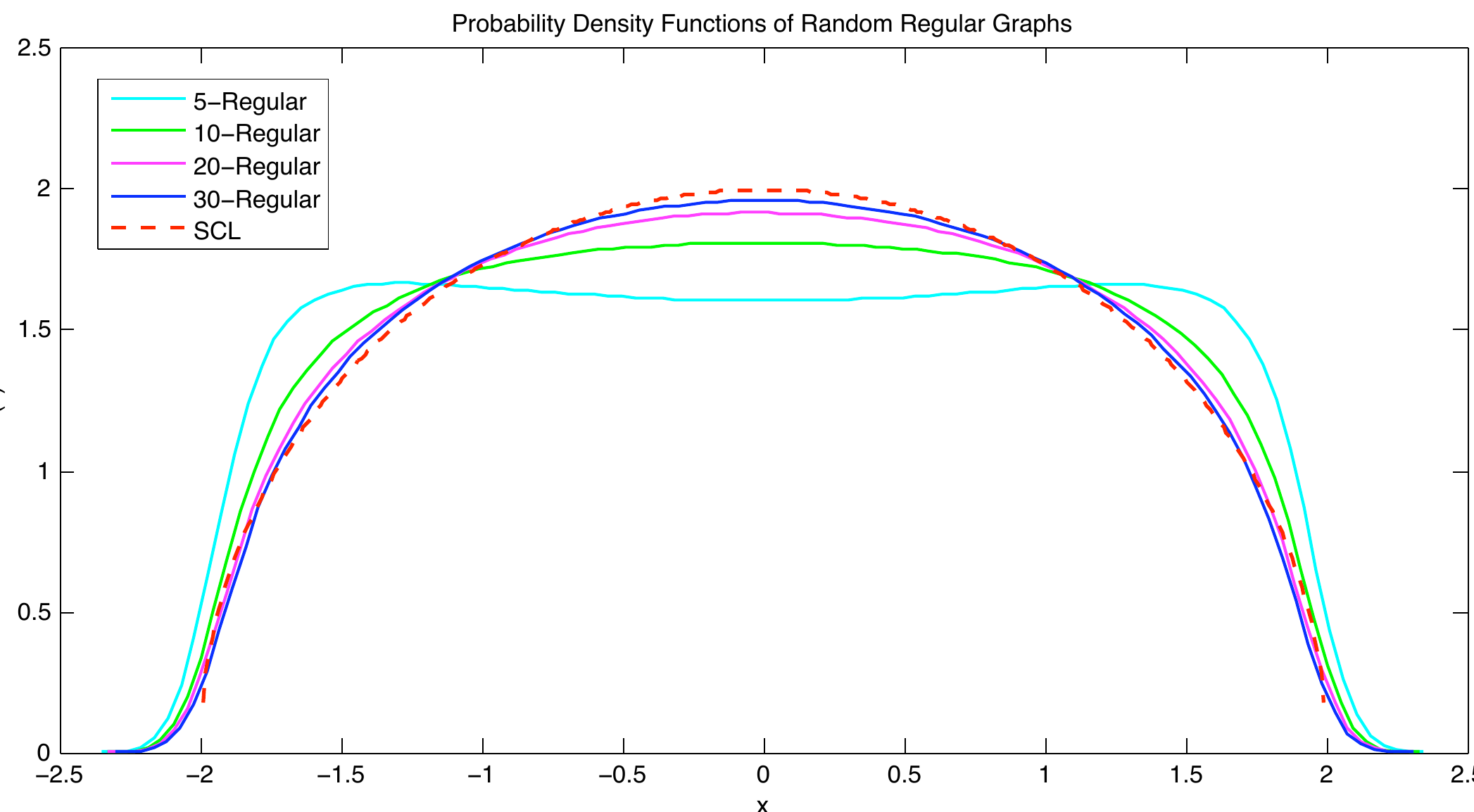}
  \captionstyle{center} 
  \onelinecaptionsfalse 
  \caption{The probability density function of the ESD of \protect\\ 
           Random $d$-regular graphs with 1000 vertices} 
\label{fig:SCL-RRG}
 \end{figure}

 It is  thus natural to conjecture that Theorem \ref{thm:SCL1} holds for $G_{n,d}$ with $d \rightarrow \infty$. 
 Let $A'_n$ be the adjacency matrix of $G_{n,d}$, and set
$$M'_n=\frac{1}{\sqrt{\frac{d}{n}(1- \frac{d}{n} ) }}(A'_n-\frac{d}{n}J).$$ 

 \begin{con} \label{conj:SCL2} 
If $d \rightarrow \infty$ then the ESD of $\frac{1}{\sqrt{n}}M'_n$ converges to the standard semicircle distribution.
 \end{con} 

Nothing has been proved about this conjecture, until recently. In \cite{SRRG}, Dimitriu and Pal showed that the conjecture holds
for $d$ tending to infinity slowly, $d= n^{o(1)}$. Their method does not extend to  larger $d$. 

We are going to establish  Conjecture \ref{conj:SCL2} in full generality. Our method is very different from that of \cite{SRRG}. 

Without loss of generality we may assume $d\le n/2$, since the adjacency matrix of the complement graph of $G_{n,d}$ may be written as $J_n-A'_n$, thus by Lemma \ref{EigenDiff} will have the spectrum interlacing between the set $\{-\lambda_n(A'_n),\dots,-\lambda_1(A'_n)\}$. Since the semi-circular distribution is symmetric, the ESD of $G_{n,d}$ will converges to semi-circular law if and only if the ESD of its complement does.

\begin{thm}\label{thm:SCL-regular} If $d$ tends to infinity with $n$, then the empirical spectral distribution of $\frac{1}{\sqrt{n}}M'_n$ converges in distribution to the  semicircle distribution.
\end{thm}

Theorem \ref{thm:SCL-regular} is a direct consequence of the following stronger result, which shows convergence at small scales. For an interval $I$ let $N'_I$ be the number of eigenvalues of $M'_n$ in $I$.

\begin{thm}\label{thm:ESD-regular}(Concentration for ESD of $G_{n,d}$). Let $\delta>0$ and consider the model $G_{n,d}$. If $d$ tends to $\infty$ as $n\rightarrow \infty$  then for any interval $I\subset[-2,2]$ with length at least $\delta^{-4/5}d^{-1/10}\log^{1/5} d$, we have
$$|N'_I-n\int_I\rho_{sc}(x)dx|< \delta n\int_I\rho_{sc}(x)dx$$
with probability at least $1-O(\exp(-cn\sqrt{d}\log d))$.
\end{thm}

\begin{rem}
Theorem \ref{thm:ESD-regular} implies that with probability $1-o(1)$, for $d=n^{\Theta(1)}$, the rank of $G_{n,d}$ is at least $n -n^{c}$ for some constant $ 0< c< 1$ (which can be computed 
explicitly from the lemmas). This  is  a partial result toward the conjecture by the second author that $G_{n,d}$ almost surely has full rank (see \cite{Vusur}). 
\end{rem}


\subsection{Infinity norm of the eigenvectors} 

Relatively little is known for eigenvectors in both random graph models under study. In \cite{DLL}, Dekel, Lee and Linial, motivated by the study of nodal domains, raised the 
 following question.
 
 \begin{question} \label{question:DLL}
Is it true that almost surely every eigenvector $u$ of $G(n, p)$ has $||u||_{\infty}=o(1)$? \end{question}

Later, in their journal paper \cite{DLL2}, the authors added one sharper question.

\begin{question} \label{question:DLL2} 

Is it true that almost surely every eigenvector $u$ of $G(n, p)$ has $||u||_{\infty}= n^{-1/2+o(1)} $? 
\end{question} 

The bound $n^{-1/2 +o(1)}$   was also conjectured by the second author of this paper in an NSF proposal (submitted Oct 2008). He and Tao \cite{tvrandom}  proved
this bound  for eigenvectors corresponding to the eigenvalues in 
 the bulk of the spectrum for the case $p=1/2$. If one defines the adjacency matrix by writting $-1$ for non-edges, then this bound holds 
for all eigenvectors \cite{tvrandom, tvrandom2}. 

The above two questions were raised under the assumption that $p$ is a constant in the interval $(0,1)$. For $p$ depending on $n$, the statements may fail. If $p \le \frac{(1-\epsilon) \log n }{n} $, then the graph has (with high probability)   isolated vertices and so one cannot expect that $\|u \| _{\infty} =o(1)$ for every eigenvector $u$. 
 We raise the following questions:

 \begin{question} \label{question:TVW1} Assume $ p \ge \frac{(1+ \epsilon) \log n }{n} $ for some constant $\epsilon >0$. 
Is it true that almost surely every eigenvector $u$ of $G(n, p)$ has $||u||_{\infty}=o(1)$? \end{question}

\begin{question} \label{question:TVW2} 
Assume $ p \ge \frac{(1+ \epsilon) \log n }{n} $ for some constant $\epsilon >0$. 
Is it true that almost surely every eigenvector $u$ of $G(n, p)$ has $||u||_{\infty}= n^{-1/2+o(1)} $? 
\end{question} 

Similarly, we can ask the above questions for $G_{n,d}$:

 \begin{question} \label{question:TVW3} Assume $ d \ge (1+ \epsilon) \log n$ for some constant $\epsilon >0$. 
Is it true that almost surely every eigenvector $u$ of $G_{n,d}$ has $||u||_{\infty}=o(1)$? \end{question}

\begin{question} \label{question:TVW4} 
Assume $ d \ge (1+ \epsilon) \log n $ for some constant $\epsilon >0$. 
Is it true that almost surely every eigenvector $u$ of $G_{n,d}$ has $||u||_{\infty}= n^{-1/2+o(1)} $? 
\end{question}


As far as random regular graphs is concerned,  Dumitriu and Pal \cite{SRRG} and Brook and Lindenstrauss \cite{BL} showed that for any normalized eigenvector of a
sparse random regular graph is delocalized in the sense that one can not have too much mass on a small set of coordinates. The readers may want to consult their papers for explicit statements. 

We generalize our questions by the following conjectures:


\begin{con}
Assume $ p \ge \frac{(1+ \epsilon) \log n }{n} $ for some constant $\epsilon >0$. Let $v$ be a random unit vector whose distribution is uniform in the $(n-1)$-dimensional unit sphere. Let $u$ be a unit eigenvector of $G(n,p)$ and $w$ be any fixed $n$-dimensional vector. Then for any $\delta>0$
$$\P(|w\cdot u-w\cdot v|>\delta)=o(1).$$ 
\end{con}

\begin{con}
Assume $ d \ge (1+ \epsilon) \log n  $ for some constant $\epsilon >0$. Let $v$ be a random unit vector whose distribution is uniform in the $(n-1)$-dimensional unit sphere. Let $u$ be a unit eigenvector of $G_{n,d}$ and $w$ be any fixed $n$-dimensional vector. Then for any $\delta>0$
$$\P(|w\cdot u-w\cdot v|>\delta)=o(1).$$ 
\end{con}

In this paper, we focus on $G(n,p)$. Our main result settles  (positively)  Question \ref{question:DLL} and almost  Question \ref{question:TVW1} . This result follows from  Corollary \ref{cor:ESDconvrate} obtained in Section 2.   

\begin{thm}
\emph{(Infinity norm of eigenvectors)}
\label{Delocal} 
Let $p =\omega(\log n /n)$ and let $A_n$ be the adjacency matrix of $G(n,p)$. Then there exists an orthonormal basis of eigenvectors of $A_n$, $\{ u_1,\ldots, u_n \}$, such that for every $1\le i \le n$, $||u_i||_{\infty}=o(1)$ almost surely. 
\end{thm}


For  Questions \ref{question:DLL2} and \ref{question:TVW2}, we obtain a good quantitative bound for those eigenvectors 
which correspond to eigenvalues bounded away from the edge of the spectrum. 

For convenience, in the case when $p=\omega(\log n/ n) \in (0,1)$, we write $$p=\frac{g(n) \log n}{n},$$ where $g(n)$ is a positive function such that $g(n)\rightarrow \infty$ as $n\rightarrow \infty$ ($g(n)$ can tend to $\infty$ arbitrarily slowly).

\begin{thm}
\label{DelocalBulk}
Assume $p={g(n)\log n}/{n} \in (0,1)$, where $g(n)$ is defined as above. Let $B_n=\frac{1}{\sqrt{n}\sigma} A_n$. For any $\kappa > 0$, and any $1 \le i \le n$ with $\lambda_i (B_n) \in [-2+\kappa, 2-\kappa]$,  there exists a corresponding eigenvector $u_i$ such that  $||u_i||_{\infty}=O_\kappa (\sqrt{ \frac{\log^{2.2} g(n)\log n}{np} } ) $with overwhelming probability.
\end{thm}

The proofs are adaptations of a recent approach developed in random matrix theory (as in \cite{tvrandom},\cite{tvrandom2},\cite{erdos2009semi}, \cite{erdos09local}). 
The main technical lemma is  a concentration theorem about the number of eigenvalues on a finer scale for $p=\omega(\log n/n) $.


\section{Semicircle law for regular random graphs}

\subsection{Proof of Theorem \ref{thm:ESD-regular}}

We use the method of comparison.  An important lemma is the following 

\begin{lem}\label{lem:np-reg} If $np\rightarrow \infty$ then $G(n,p)$ is $np$-regular with probability at least $\exp(-O(n(np)^{1/2})$.
\end{lem}

For the range $p\ge\log^2 n/n$, Lemma \ref{lem:np-reg} is a consequence of 
 a result of Shamir and Upfal \cite{SU} (see also \cite{KSVW}). For smaller values of $np$, McKay and Wormald \cite{MKW1} calculated
 precisely the probability that $G(n,p)$ is $np$-regular, using the fact that the joint distribution of the degree sequence of $G(n,p)$ can be approximated by a simple model derived from  independent random variables with binomial distribution. Alternatively, one may calculate the same probability directly using the asymptotic formula for the number of $d$-regular graphs on $n$ vertices (again by McKay and Wormald \cite{MKW2}). Either way, for $p=o(1/\sqrt{n})$, we know that 
$$\P(G(n,p)\text{ is }np\text{-regular})\ge \Theta(\exp(-n\log(\sqrt{np})).$$
which is better than claimed in  Lemma \ref{lem:np-reg}.\\

Another  key ingredient  is the following concentration lemma, which may be of independent interest. 

\begin{lem}\label{thm:ESD-con-general} Let $M$ be a $n\times n$ Hermitian random matrix whose off-diagonal entries $\xi_{ij}$ are i.i.d. random variables with mean zero, variance 1 and $|\xi_{ij}|< K$ for some common constant $K$. Fix $\delta>0$ and assume that the forth moment  $M_4:=\sup_{i,j}\E(|\omega_{ij}|^4)=o(n)$. Then for any interval $I\subset [-2,2]$ whose length is at least $\Omega(\delta^{-2/3}(M_4/n)^{1/3})$, the number $N_I$ of the eigenvalues of $\frac{1}{\sqrt{n}}M$ which belong to $I$ satisfies the following concentration inequality
$$\P(|N_I-n\int_I\rho_{sc}(t)dt|>\delta n\int_I\rho_{sc}(t)dt)\le 4\exp(-c\frac{\delta^4n^2|I|^5}{K^2}).$$
\end{lem}


Apply Lemma  \ref{thm:ESD-con-general} for the normalized adjacency matrix $M_n$ of $G(n,p)$ with $K=1/\sqrt{p}$  we obtain 

\begin{cor}\label{cor:ESDconvrate} Consider the model $G(n,p)$ with  $np\rightarrow  \infty$ as $n\rightarrow  \infty$  and let $\delta>0$. Then for any interval $I\subset[-2,2]$ with length at least $\big(\frac{\log(np)}{\delta^4(np)^{1/2}}\big)^{1/5}$, we have
$$|N_I-n\int_I\rho_{sc}(x)dx|\ge \delta n\int_I\rho_{sc}(x)dx$$
with probability at most $\exp(-cn(np)^{1/2}\log(np))$.
\end{cor}

\begin{rem} If one only needs the result for the bulk case $I\subset [-2+\epsilon,2-\epsilon]$ for an absolute constant $\epsilon>0$ then the minimum length of $I$ can be improved to $\big(\frac{\log(np)}{\delta^4(np)^{1/2}}\big)^{1/4}$.
\end{rem}

By Corollary \ref{cor:ESDconvrate} and Lemma \ref{lem:np-reg}, the probability that $N_I$ fails to be close to the expected value in the model $G(n,p)$ is much smaller than the probability that $G(n,p)$ is $np$-regular. Thus the probability that $N_I$ fails to be close to the expected value in the model $G_{n,d}$ where $d=np$ is the ratio of the two former probabilities, which is $O(\exp(-cn\sqrt{np}\log np))$ for some small positive constant $c$. Thus, Theorem \ref{thm:ESD-regular} is proved, depending on Lemma \ref{thm:ESD-con-general} which we turn to next.


\subsection
{Proof of Lemma  \ref{thm:ESD-con-general}} 

Assume $I=[a,b]$ and $a-(-2)<2-b$. 

We will use the approach of Guionnet and Zeitouni in \cite{GZ}. Consider a random Hermitian matrix $W_n$ with independent entries $w_{ij}$ with support in a compact region $S$. Let $f$ be a real convex $L$-Lipschitz function and define
$$Z:=\sum_{i=1}^nf(\lambda_i)$$
where $\lambda_i$'s are the eigenvalues of $\frac{1}{\sqrt{n}}W_n$. We are going to view $Z$ as the function of the atom variables $w_{ij}$. For our application we need $w_{ij}$ to be random variables with mean zero and variance 1, whose absolute values are bounded by a common constant $K$.

The following concentration inequality is from \cite{GZ}

\begin{lem}\label{GZconcentration} Let $W_n, f, Z$ be as above. Then there is a constant $c>0$ such that for any $T>0$
$$\P(|Z-\E(Z)|\ge T)\le 4\exp(-c\frac{T^2}{K^2L^2}).$$
\end{lem}

In order to apply Lemma \ref{GZconcentration} for $N_I$ and $M$, it is natural to consider 
$$Z:=N_I=\sum_{i=1}^n \chi_I(\lambda_i)$$
where $\chi_I$ is the indicator function of $I$ and $\lambda_i$ are the eigenvalues of $\frac{1}{\sqrt{n}}M_n$. However, this function is neither convex nor Lipschitz. As suggested in \cite{GZ}, one can  overcome this problem by a proper approximation.  Define $I_l=[a-\frac{|I|}{C},a]$, $I_r=[b,b+\frac{|I|}{C}]$ and construct two real functions $f_1, f_2$ as follows(see Figure \ref{fig:fg}): 
\begin{equation*}
f_1(x)=\Bigg\{
\begin{array}{ll}
-\frac{C}{|I|}(x-a)-1& \text{if }x\in (-\infty, a-\frac{|I|}{C})\\
0&\text{if }x\in I\cup I_l\cup I_r\\
\frac{C}{|I|}(x-b)-1& \text{if }x\in (b+\frac{|I|}{C},\infty)
\end{array}
\end{equation*}

\begin{equation*}
f_2(x)=\Bigg\{
\begin{array}{ll}
-\frac{C}{|I|}(x-a)-1& \text{if }x\in (-\infty, a)\\
-1&\text{if }x\in I\\
\frac{C}{|I|}(x-b)-1& \text{if }x\in (b,\infty)
\end{array}
\end{equation*}

 where $C$ is a constant to be chosen later. Note that $f_j$'s are convex and $\frac{C}{|I|}$-Lipschitz. Define $$X_1=\sum_{i=1}^n f_1(\lambda_i),\ X_2=\sum_{i=1}^n f_2(\lambda_i)$$ and  apply Lemma \ref{GZconcentration} with $T=\frac{\delta}{8} n\int_I\rho_{sc}(t)dt$ for $X_1$ and $X_2$. Thus, we have
\begin{align*}
\P(|X_j-\E(X_j)|\ge \frac{\delta}{8} n\int_I\rho_{sc}(t)dt)&\le4\exp(-c\frac{\delta^2n^2|I|^2(\int_I\rho_{sc}(t)dt)^2}{K^2C^2}).\\
\end{align*}
At this point we need to estimate the value of $\int_I\rho_{sc}(t)dt$. There are two cases: if $I$ is in the ``bulk'' i.e. $I\subset[-2+\epsilon,2-\epsilon]$ for some positive absolute constant $\epsilon$, then $\int_I\rho_{sc}(t)dt=\alpha |I|$ where $\alpha$ is a constant depending on $\epsilon$. But if $I$ is very near the edge of $[-2,2]$ i.e. $a-(-2)<|I|=o(1)$, then $\int_I\rho_{sc}(t)dt=\alpha'|I|^{3/2}$ for some absolute constant $\alpha'$. Thus in both case we have
$$\P(|X_j-\E(X_j)|\ge \frac{\delta}{8} n\int_I\rho_{sc}(t)dt)\le4\exp(-c_1\frac{\delta^2n^2|I|^5}{K^2C^2})$$

 Let $X=X_1-X_2$, then 
$$\P(|X-\E(X)|\ge \frac{\delta}{4} n\int_I\rho_{sc}(t)dt)\le O(\exp(-c_1\frac{\delta^2n^2|I|^5}{K^2C^2})).$$

Now we compare $X$ to $Z$, making use of a result of 
G\"otze and Tikhomirov \cite{GT}.  We have $\E(X-Z)\le \E(N_{I_l}+N_{I_r})$. In \cite{GT}, G\"{o}tze and Tikhomirov obtained  a convergence rate for ESD of Hermitian random matrices whose entries have mean zero and variance one, which implies that for any $I\subset[-2,2]$
$$|\E(N_I)-n\int_I\rho_{sc}(t)dt|<\beta n\sqrt{\frac{M_4}{n}},$$
where $\beta$ is an absolute constant, $M_4=\sup_{i,j}\E(|\omega_{ij}|^4)$. 
Thus
$$\E(X)\le\E(Z)+n\int_{I_l\cup I_r}\rho_{sc}(t)dt+\beta n\sqrt{\frac{M_4}{n}}.$$
In the ``edge" case we can choose $C=(4/\delta)^{2/3}$, then because $|I|\ge \Omega(\delta^{-2/3}(M_4/n)^{1/3})$, we have
$$n\int_{I_l\cup I_r}\rho_{sc}(t)dt=\Theta(n(\frac{|I|}{C})^{3/2})>\Omega( n\sqrt{\frac{M_4}{n}})$$
and
$$n\int_{I_l\cup I_r}\rho_{sc}(t)dt+\beta n\sqrt{\frac{M_4}{n}}=\Theta(n(\frac{|I|}{C})^{3/2})=\Theta(\frac{\delta}{4}n\int_I\rho_{sc}(t)dt).$$
In the ``bulk'' case we choose $C=4/\delta$, then
$$n\int_{I_l\cup I_r}\rho_{sc}(t)dt+\beta n\sqrt{\frac{M_4}{n}}=\Theta(n\frac{|I|}{C})=\Theta(\frac{\delta}{4}n\int_I\rho_{sc}(t)dt).$$
Therefore in both cases, with probability at least $1-O(\exp(-c_1\frac{\delta^4n^2|I|^5}{K^2}))$, we have
$$Z\le X\le \E(X)+ \frac{\delta}{4}n\int_I\rho_{sc}(t)dt < \E(Z)+\frac{\delta}{2} n\int_I\rho_{sc}(t)dt.$$
The convergence rate result of G\"{o}tze and Tikhomirov again gives
$$\E(N_I)<n\int_I\rho_{sc}(t)dt+ \beta n\sqrt{\frac{M_4}{n}}<(1+\frac{\delta}{2})n\int_I\rho_{sc}(t)dt,$$
hence with probability at least $1-O(\exp(-c_1\frac{\delta^4n^2|I|^5}{K^2}))$
$$Z<(1+\delta) n\int_I\rho_{sc}(t)dt,$$
which is the desires upper bound.

The lower bound is proved using a similar argument. Let $I'=[a+\frac{|I|}{C}, b-\frac{|I|}{C}]$, $I'_l=[a,a+\frac{|I|}{C}]$, $I'_r=[b-\frac{|I|}{C},b]$ where $C$ is to be chosen later and define two functions $g_1$, $g_2$ as follows (see Figure \ref{fig:fg}):

\begin{equation*}
g_1(x)=\Bigg\{
\begin{array}{ll}
-\frac{C}{|I|}(x-a)& \text{if }x\in (-\infty, a)\\
0&\text{if }x\in I'\cup I'_l\cup I'_r\\
\frac{C}{|I|}(x-b)& \text{if }x\in (b,\infty)
\end{array}
\end{equation*}

\begin{equation*}
g_2(x)=\Bigg\{
\begin{array}{ll}
-\frac{C}{|I|}(x-a)& \text{if }x\in (-\infty, a+\frac{|I|}{C})\\
-1&\text{if }x\in I'\\
\frac{C}{|I|}(x-b)& \text{if }x\in (b-\frac{|I|}{C},\infty)
\end{array}
\end{equation*}

Define $$Y_1=\sum_{i=1}g_1(\lambda_i),\ Y_2=\sum_{i=1}g_2(\lambda_i).$$ Applying Lemma \ref{GZconcentration} with $T=\frac{\delta}{8} n\int_{I}\rho_{sc}(t)dt$ for $Y_j$ and using the estimation for $\int_I\rho(t)dt$ as above, we have
$$\P(|Y_j-\E(Y_j)|\ge \frac{\delta}{8} n\int_{I}\rho_{sc}(t)dt)\le4\exp(-c_2\frac{\delta^2n^2|I|^5}{K^2C^2}).$$
Let $Y=Y_1-Y_2$, then 
$$\P(|Y-\E(Y)|\ge \frac{\delta}{4} n\int_{I}\rho_{sc}(t)dt)\le O(\exp(-c_2\frac{\delta^2n^2|I|^5}{K^2C^2})).$$

We have $\E(Z-Y)\le \E(N_{I'_l}+N_{I'_r})$.  A similar argument as in the proof of the upper bound (using the convergence rate of G\"{o}tze and Tikhomirov) shows

$$\E(Y)\ge\E(Z)-n\int_{I'_l\cup I'_r}\rho_{sc}(t)dt-\beta n\sqrt{\frac{M_4}{n}}>E(Z)- \frac{\delta}{4}n\int_I\rho_{sc}(t)dt.$$
Therefore with probability at least $1-O(\exp(-c_2\frac{\delta^2n^2|I|^5}{K^2C^2}))$, we have
$$Z\ge Y\ge \E(Y)- \frac{\delta}{4}n\int_I\rho_{sc}(t)dt > \E(Z)-\frac{\delta}{2} n\int_I\rho_{sc}(t)dt,$$
and by the convergence rate, with probability at least $1-O(\exp(-c2\frac{\delta^2n^2|I|^5}{K^2C^2}))$
$$Z>(1-\delta)n\int_I\rho_{sc}(t)dt.$$
Thus, Theorem \ref{thm:ESD-con-general} is proved.
\endproof

\begin{figure} [htbp]
  \centering 
  \includegraphics[scale=0.52]{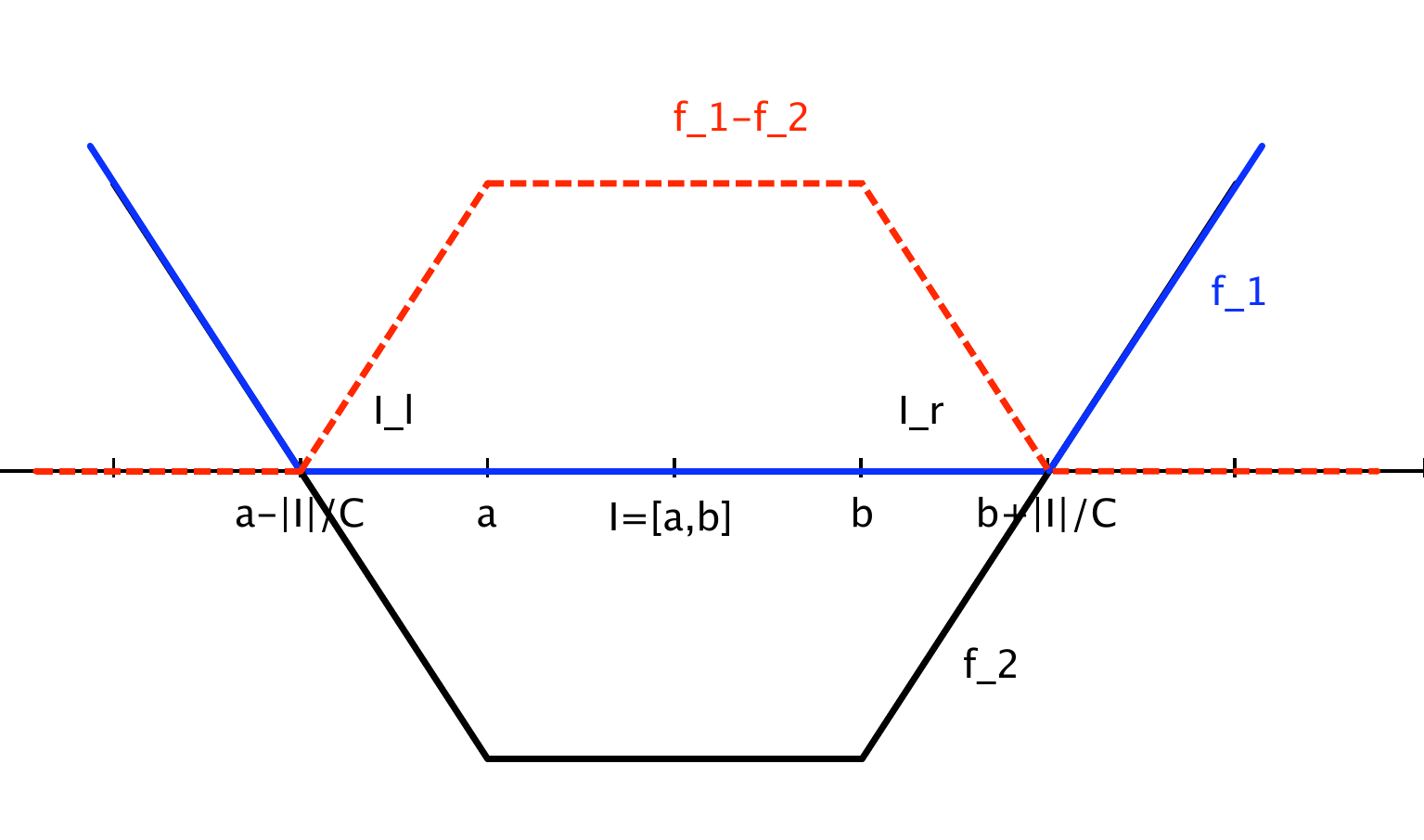}%
  \hfill{}%
  \includegraphics[scale=0.52]{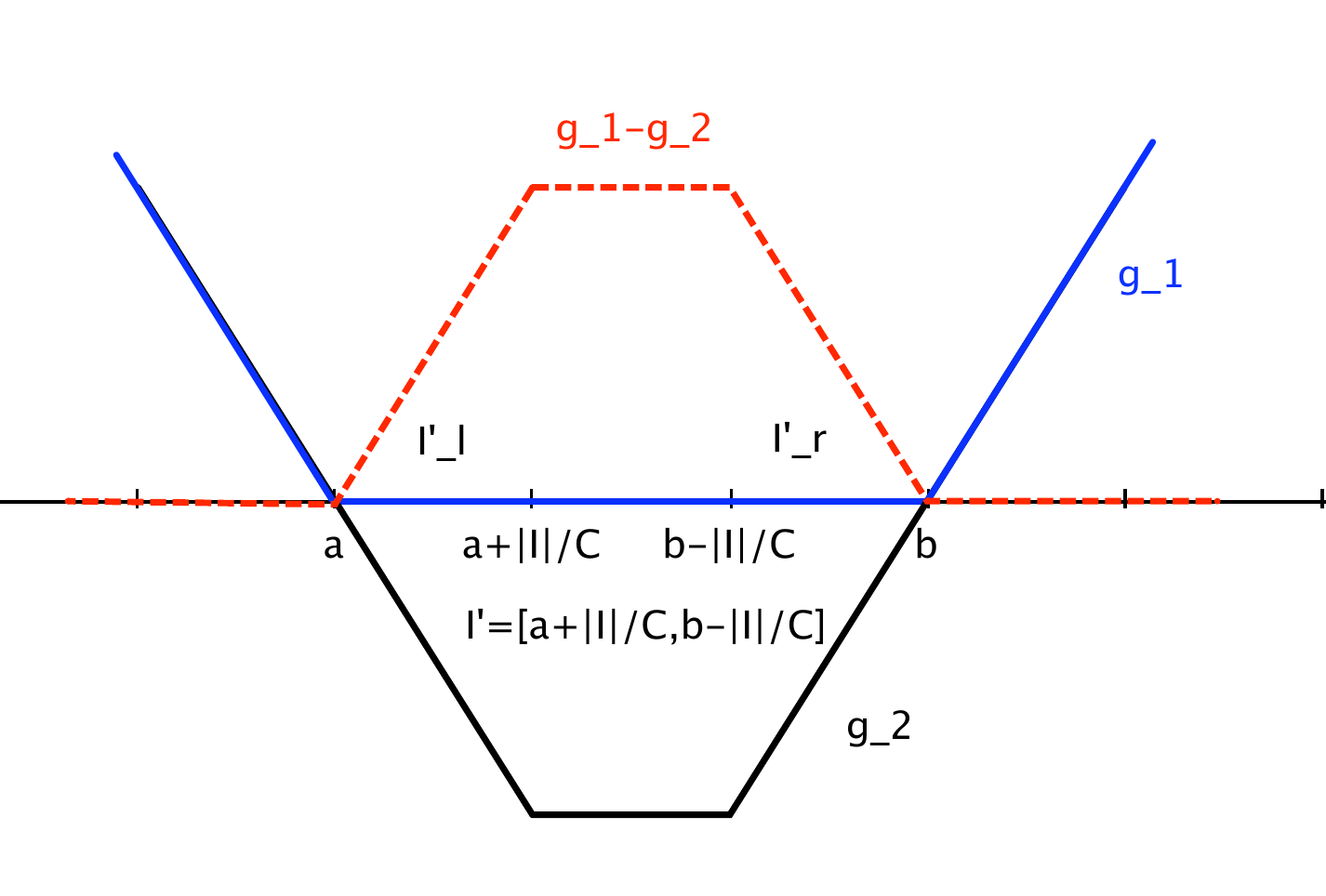} 
  \caption{Auxiliary functions used in the proof} 
  \label{fig:fg}
\end{figure}


\section{Infinity norm of the eigenvectors}

\subsection{Small perturbation lemma}

$A_n$ is the adjacency matrix of $G(n,p)$. In the proofs of Theorem \ref{Delocal} and Theorem \ref{DelocalBulk}, we actually work with the eigenvectors of a perturbed matrix $$A_n+\epsilon N_n,$$ where $\epsilon=\epsilon(n) >0$ can be arbitrarily small and $N_n$ is a symmetric random matrix whose upper triangular elements are independent with a standard Gaussian distribution. 

The entries of $A_n+\epsilon N_n$ are continuous and thus with probability 1, the eigenvalues of $A_n+\epsilon N_n$ are simple. Let $$\mu_1 <\ldots <\mu_n$$ be the ordered eigenvalues of $A_n+\epsilon N_n$, which have a unique orthonormal system of eigenvectors $\{w_1,\ldots,w_n\}$. By the Cauchy interlacing principle, the eigenvalues of $A_n+\epsilon N_n$ are different from those of its principle minors, which satisfies a condition of Lemma \ref{EigenEntry}.

Let $\lambda_i$'s be the eigenvalue of $A_n$ with multiplicity $k_i$ defined as follows: $$\ldots \lambda_{i-1} < \lambda_{i}=\lambda_{i+1}=\ldots=\lambda_{i+k_i}<\lambda_{i+k_i+1}\ldots$$

By Weyl's theorem, one has for every $1\le j\le n$,

\begin{equation} \label{eq:weyl}
|\lambda_j - \mu_j | \le \epsilon || N_n ||_{\text{op}} = O(\epsilon \sqrt{n})
\end{equation}

Thus the behaviors of eigenvalues of $A_n$ and $A_n+\epsilon N_n$ are essentially the same by choosing $\epsilon$ sufficiently small. And everything (except Lemma \ref{EigenEntry}) we used in the proofs of Theorem \ref{Delocal} and Theorem \ref{DelocalBulk} for $A_n$ also applies for $A_n+\epsilon N_n$ by a continuity argument. We will not distinguish $A_n$ from $A_n+\epsilon N_n$ in the proofs.  

The following lemma will allow us to transfer the eigenvector delocaliztion results of $A_n+\epsilon N_n$ to those of $A_n$ at some expense.

\begin{lem}\label{lem:perturb}
In the notations of above, there exists an orthonormal basis of eigenvectors of $A_n$, denoted by $\{u_1,\ldots,u_n \}$, such that for every $1\le j \le n$, $$||u_j||_{\infty} \le || w_j ||_{\infty} +\alpha(n),$$
where $\alpha(n)$ can be arbitrarily small provided $\epsilon(n)$ is small enough. 
\end{lem}

\begin{proof}
First, since the coefficients of the characteristic polynomial of $A_n$ are integers, there exists a positive function $l(n)$ such that either $|\lambda_s- \lambda_t|=0$ or $|\lambda_s-\lambda_t| \ge l(n)$ for any $1\le s,t \le n$.

By (\ref{eq:weyl}) and choosing $\epsilon$ sufficiently small, one can get 
$$|\mu_i- \lambda_{i-1}| > l(n) ~~ \text{and} ~~ | \mu_{i+k_i} -\lambda_{i+k_i+1}| > l(n)$$

For a fixed index $i$, let $E$ be the eigenspace corresponding to the eigenvalue $\lambda_i$ and $F$ be the subspace spanned by $\{w_i,\ldots,w_{i+k_i} \}$. Both of $E$ and $F$ have dimension $k_i$. Let $P_E$ and $P_F$ be the orthogonal projection matrices onto $E$ and $F$ separately.

Applying the well-known Davis-Kahan theorem (see \cite{SS} Section IV, Theorem 3.6) to $A_n$ and $A_n+\epsilon N_n$, one gets 
$$|| P_E- P_F ||_{\text{op} } \le \frac{\epsilon || N_n ||_{\text{op} }} {l(n)} := \alpha(n),$$
where $\alpha(n)$ can be arbitrarily small depending on $\epsilon.$

Define $v_j=P_F w_j \in E$ for $i \le j \le i+k_i$, then we have $||v_j- w_j||_2 \le \alpha(n)$. It is clear that $\{v_i,\ldots,v_{k_i}\}$ are eigenvectors of $A_n$ and 
$$||v_j||_{\infty} \le ||w_j||_{\infty} + ||v_j - w_j||_2 \le ||w_j||_{\infty} + \alpha(n).$$

By choosing $\epsilon$ small enough such that $n\alpha(n) < 1/2$, $\{v_i,\ldots,v_{k_i}\}$ are linearly independent. Indeed, if $\sum_{j=i}^{k_i} c_j v_j=0$, one has for every $i \le s \le i+k_i$, $\sum_{j=i}^{k_i} c_j \langle P_F w_j , w_s \rangle=0$, which implies $c_s = -\sum_{j=i}^{k_i} c_j \langle P_F w_j - w_j, w_s\rangle$. Thus 
$|c_s| \le \alpha(n) \sum_{j=i}^{k_i} |c_j|,$ summing over all $s$, we can get $\sum_{j=i}^{k_i} |c_j| \le k\alpha(n) \sum_{j=i}^{k_i} |c_j|$ and therefore $ c_j=0$.

Furthermore the set $\{v_i,\ldots,v_{k_i}\}$ is 'almost' an orthonormal basis of $E$ in the sense that 

\begin{equation*}
\begin{split}
|~ ||v_s||_2 -1 ~| & \le ||v_s- w_s||_2 \le \alpha(n)  ~~~~~\text{for any $i \le s \le i+k_i$ }\\
\\
|\langle v_s, v_t \rangle| &=|\langle P_F w_s, P_F w_t\rangle| \\
&=|\langle P_F w_s -w_s, P_F w_t\rangle + \langle w_s, P_F w_t - w_t\rangle | \\
&= O( \alpha(n) ) ~~~~~~~~\text{for any $i \le s\neq t \le i+k_i$ }\\
\end{split}
\end{equation*}

We can perform a Gram-Schmidt process on $\{v_i,\ldots,v_{k_i}\}$ to get an orthonormal system of eigenvectors $\{u_i,\ldots,u_{k_i} \}$ on $E$ such that $$||u_j||_{\infty} \le || w_j ||_{\infty} +\alpha(n),$$
for every $i \le j \le i+k_i$.

We iterate the above argument for every distinct eigenvalue of $A_n$ to obtain an orthonormal basis of eigenvectors of $A_n$.

\end{proof}



\subsection{Auxiliary lemmas}

\begin{lem} 
\emph{(Lemma 41, \cite{tvrandom})}
\label{EigenEntry}
Let $$B_n=
\left(
\begin{array}{cc}
a & X^* \\
X & B_{n-1}
\end{array}
\right)
 $$
be a $n\times n$ symmetric matrix for some $a\in \mathbb{C}$ and $X \in \mathbb{C}^{n-1}$, and let 
$\left(
\begin{array}{cc}
x\\
v
\end{array}
\right)$ be a eigenvector of $B_n$ with eigenvalue $\lambda_i(B_n)$, where $x\in \mathbb{C}$ and $v \in \mathbb{C}^{n-1}$. Suppose that none of the eigenvalues of $B_{n-1}$ are equal to $\lambda_i(B_n)$. Then 
$$|x|^2= \frac{1}{1+\sum_{j=1}^{n-1} (\lambda_j (B_{n-1}) - \lambda_i(B_n))^{-2} |u_j (B_{n-1})^* X|^2},$$
where $u_j(B_{n-1})$ is a unit eigenvector corresponding to the eigenvalue $\lambda_j (B_{n-1}).$
\end{lem}

The \textit{Stieltjes transform} $s_n(z)$ of a symmetric matrix $W$ is defined for $z \in \mathbb{C}$ by the formula 
$$s_n(z):=\frac{1}{n} \displaystyle\sum_{i=1}^{n} \frac{1}{\lambda_i(W)-z}.$$ It has the following alternate representation:

\begin{lem} 
\emph{(Lemma 39, \cite{tvrandom})}
\label{StieTran}
Let $W=(\zeta_{ij})_{1\le i,j\le n}$ be a symmetrix matrix, and let $z$ be a complex number not in the spectrum of $W$. Then we have 
$$s_n(z)=\frac{1}{n} \displaystyle\sum_{k=1}^{n} \frac{1}{\zeta_{kk}-z- a^*_k (W_k -zI)^{-1} a_k }$$
where $W_k$ is the $(n-1) \times (n-1)$ matrix with the $k^\text{th}$ row and column of $W$ removed, and $a_k \in \mathbb{C}^{n-1}$ is the $k^\text{th}$ column of $W$ with the $k^\text{th}$ entry removed.
\end{lem}

We begin with two lemmas that will be needed to prove the main results. The first lemma, following the paper \cite{tvrandom} in Appendix B, uses Talagrand's inequality. Its proof is presented in the Appendix \ref{appendix:projection}.

\begin{lem} 
\label{ConcenLem}
Let $Y=(\zeta_1,\ldots,\zeta_n) \in \mathbb{C}^n$ be a random vector whose entries are i.i.d. copies of the random variable $\zeta=\xi-p$ (with mean $0$ and variance $\sigma^2$). Let $H$ be a subspace of dimension $d$ and $\pi_H$ the orthogonal projection onto H. Then 

$$\textbf{P}(|\parallel \pi_H (Y) \parallel -\sigma \sqrt{d} |\ge t) \le 10 \exp(-\frac{t^2}{4}).$$
In particular,
\begin{equation} \label{eq:1.1}
\parallel \pi_H(Y) \parallel= \sigma \sqrt{d}+O( \omega(\sqrt{\log n}) )
\end{equation}
with overwhelming probability.
\end{lem}

The following concentration lemma for $G(n,p)$ will be a key input to prove Theorem \ref{DelocalBulk}.
Let $B_n=\frac{1}{\sqrt{n}\sigma} A_n$
\begin{lem}[Concentration for ESD in the bulk]
\emph{(Concentration for ESD in the bulk)} 
\label{ConBulkAdj}
Assume $p={g(n)\log n}/{n}$. For any constants $\varepsilon, \delta > 0$ and any interval $I$ in $[-2+\varepsilon, 2-\varepsilon]$ of width $|I|=\Omega( {\log^{2.2} g(n) \log n}/{np} )$, the number of eigenvalues $N_I$ of $B_n$ in $I$ obeys the concentration estimate $$|N_I(B_n) - n \displaystyle\int_I {{}\rho}_{sc}(x)\,dx| \le {\delta} n |I|$$ with overwhelming probability. 
\end{lem}

The above lemma is a variant of Corollary \ref{cor:ESDconvrate}. This lemma allows us to control the ESD on a smaller interval and the proof, relying on a projection lemma (Lemma \ref{ConcenLem}), is a different approach. The proof is presented in Appendix \ref{appendix:concentration}.


\subsection{Proof of Theorem \ref{Delocal}:}

Let $\lambda_n(A_n)$ be the largest eigenvalue of $A_n$ and $u=(u_1,\ldots,u_n)$ be the corresponding unit eigenvector. We have the lower bound $\lambda_n(A_n) \ge  np$. And if $np=\omega(\log n)$, then the maximum degree $\Delta = (1+o(1))np$ almost surely (See Corollary 3.14, \cite{bollobas}).

For every $1\le i \le n$, $$\lambda_n(A_n) u_i = \sum_{j \in N(i)} u_j,$$ 
where $N(i)$ is the neighborhood of vertex $i$. Thus, by Cauchy-Schwarz inequality,

$$|| u ||_{\infty}=\text{max}_i \frac{| \sum_{j \in N(i)} u_j |}{\lambda_n(A_n)} \le \frac{\sqrt{ \Delta } }{\lambda_n(A_n) } = O(\frac{1}{\sqrt{np} }).$$

Let $B_n=\frac{1}{\sqrt{n} \sigma} A_n$. Since the eigenvalues of $W_n=\frac{1}{\sqrt{n} \sigma } (A_n - p J_n)$ are on the interval $[-2,2]$, by Lemma \ref{EigenDiff}, $\{\lambda_1(B_n), \ldots, \lambda_{n-1}(B_n) \} \subset [-2,2] $. 

Recall that $np = g(n)\log n$. By Corollary \ref{cor:ESDconvrate}, for any interval $I$ with length at least $(\frac{\log (np)}{{\delta}^4 (np)^{1/2}})^{1/5}$(say $\delta=0.5$),with overwhelming probability, if $I  \subset [-2+\kappa,2-\kappa]$ for some positive constant $\kappa$, one has $N_I(B_n)= \Theta(n \int_{I} \rho_{sc}(x) dx) = \Theta(n |I|)$; if $I$ is at the edge of $[-2,2]$, with length $o(1)$, one has $N_I(B_n)=\Theta(n\int_{I} \rho_{sc}(x) dx) = \Theta(n |I|^{3/2})$. Thus we can find a set $J \subset \{1,\ldots,n-1\}$ with $|J| =\Omega( n |I_0|$) or $|J| =\Omega(n |I_0|^{3/2})$ such that $|\lambda_j(B_{n-1})-\lambda_i(B_n)| \ll |I_0|$ for all $j \in J$, where $B_{n-1} $ is the bottom right $(n-1) \times (n-1)$ minor of $B_n$. Here we take $|I_0|=(1/g(n)^{1/20})^{2/3}$. It is easy to check that $|I_0| \ge (\frac{\log (np)}{{\delta}^4 (np)^{1/2}})^{1/5}$.

By the formula in Lemma \ref{EigenEntry}, the entry of the eigenvector of $B_n$ can be expressed as 

\begin{equation} \label{eq:entry}
\begin{split}
|x|^2 &=\displaystyle\frac{1}{1+\sum_{j=1}^{n-1} (\lambda_j (B_{n-1}) - \lambda_i(B_n))^{-2} |u_j (B_{n-1})^* \frac{1}{\sqrt{n}\sigma}X|^2} \\
&\le \frac{1}{1+\sum_{j \in J} (\lambda_j (B_{n-1}) - \lambda_i(B_n))^{-2} |u_j (B_{n-1})^* \frac{1}{\sqrt{n}\sigma}X|^2} \\
&\le  \frac{1}{1+\sum_{j \in J} n^{-1}|I_0|^{-2} |u_j (B_{n-1})^* \frac{1}{\sigma}X|^2} = \frac{1}{1+ n^{-1}|I_0|^{-2} ||\pi_{H}(\frac{X}{\sigma})||^2}\\
&\le \frac{1}{1+{n^{-1}|I_0|^{-2}}{|J|}}
\end{split}
\end{equation}
with overwhelming probability, where $H$ is the span of all the eigenvectors associated to $J$ with dimension $\text{dim}(H)=\Theta(|J|)$, $\pi_{H}$ is the orthogonal projection onto $H$ and $X \in \mathbb{C}^{n-1}$ has entries that are iid copies of $\xi$.  The last inequality in (\ref{eq:entry}) follows from 
Lemma \ref{ConcenLem} (by taking $t=g(n)^{1/10}\sqrt{\log n}$) and the relations $$||\pi_H(X)||=||\pi_H(Y+p \bold{1}_n)|| \ge ||\pi_{H_1} (Y+p\bold{1}_n)|| \ge ||\pi_{H_1} (Y)|| .$$ Here $Y=X-p\bold{1}_n$ and $H_1=H \cap H_2$, where $H_2$ is the space orthogonal to the all 1 vector $\bold{1}_n$. For the dimension of $H_1$, $\text{dim}(H_1)\ge \text{dim}(H)-1$ .

\medskip

Since either $|J|=\Omega (n |I_0|$) or $|J| =\Omega( n |I_0|^{3/2})$,  we have ${n^{-1}|I_0|^{-2}}{|J|} =\Omega( {|I_0|}^{-1}$) or ${n^{-1}|I_0|^{-2}}{|J|} =\Omega( {|I_0|}^{-1/2}$). Thus $|x|^2 =O( |I_0|)$ or $|x|^2 =O( \sqrt{|I_0|})$. In both cases, since $|I_0| \rightarrow 0$, it follows that $|x|=o(1)$.
\hfill $\Box$


\subsection{Proof of Theorem \ref{DelocalBulk}}
With the formula in Lemma \ref{EigenEntry}, it suffices to show the following lower bound

\begin{equation} \label{eq:1.11}
\sum_{j=1}^{n-1} (\lambda_j (B_{n-1}) - \lambda_i(B_n))^{-2} |u_j (B_{n-1})^* \frac{1}{\sqrt{n}\sigma}X|^2 \gg \frac{np}{\log^{2.2} g(n) \log n}
\end{equation}
with overwhelming probability, where $B_{n-1} $ is the bottom right $n-1 \times n-1$ minor of $B_n$ and $X \in \mathbb{C}^{n-1}$ has entries that are iid copies of $\xi$. Recall that $\xi$ takes values $1$ with probability $p$ and $0$ with probability $1-p$, thus $\mathbb{E} \xi=p, \mathbb{V}ar{\xi}=p(1-p)={\sigma}^2$.

\medskip
By Theorem \ref{ConBulkAdj}, we can find a set $J \subset \{1,\ldots,n-1\}$ with $|J| \gg \frac{\log^{2.2} g(n)\log n}{p}$ such that $|\lambda_j(B_{n-1})-\lambda_i(B_n)| =O(\log^{2.2} g(n)\log n/{np})$ for all $j \in J$. Thus in (\ref{eq:1.11}), it is enough to prove
$$\displaystyle\sum_{j \in J} |u_j(B_{n-1})^T \frac{1}{\sigma}X|^2= ||\pi_{H}(\frac{X}{\sigma})||^2 \gg  |J| $$
or equivalently
\begin{equation}
||\pi_{H}(X)||^2 \gg  {\sigma}^2 |J| 
\end{equation}
with overwhelming probability, where $H$ is the span of all the eigenvectors associated to $J$ with dimension $\text{dim}(H)=\Theta(|J|)$.

\medskip
Let $H_1=H\cap H_2$, where $H_2$ is the space orthogonal to $\bold{1}_n$. The dimension of $H_1$ is at least $\text{dim}(H)-1$.
Denote $Y=X- p \bold{1}_n$. Then the entries of $Y$ are iid copies of $\zeta$. By Lemma \ref{ConcenLem}, $$||\pi_{H_1}(Y)||^2 \gg {\sigma}^2 |J|$$ with overwhelming probability. 

\medskip
Hence, our claim follows from the relations
$$||\pi_H(X)||=||\pi_H(Y+p \bold{1}_n)|| \ge ||\pi_{H_1} (Y+p\bold{1}_n)|| = ||\pi_{H_1} (Y)||.$$
\hfill $\Box$


\begin{appendices}

In this appendix, we complete the proofs of Theorem \ref{thm:SCL1}, Lemma \ref{ConcenLem} and Lemma \ref{ConBulkAdj}.


\section{Proof of Theorem \ref{thm:SCL1}}\label{appendix:SCL}
\medskip
We will show that the semicircle law holds for $M_n$. With Lemma \ref{EigenDiff}, it is clear that Theorem \ref{thm:SCL1} follows Lemma \ref{ESD} directly.
The claim actually follows as a special case discussed in the paper \cite{ch04spectra}. Our proof here uses a standard moment method.

\begin{lem}
\label{ESD}
For $p=\omega(\frac{1}{n})$, the empirical spectral distribution (ESD) of the matrix $W_n=\frac{1}{\sqrt{n}} M_n$ converges in distribution to the semicircle law which has a density ${{}\rho}_{sc}(x)$ with support on $[-2,2]$,  $${{\rho}}_{sc}(x) := \frac{1}{2 \pi } \sqrt{4 -x^2}.$$
\end{lem}

Let ${\eta}_{ij} $ be the entries of $M_n={\sigma}^{-1} (A_n-p J_n)$.  For $i = j$, $\eta_{ij}=-p/\sigma$; and for $i\not= j$, $\eta_{ij}$ are iid copies of random variable $\eta$, which takes value $(1-p)/ \sigma$ with probability $p$ and takes value $-p/ \sigma$ with probability $1-p$. 

$$\textbf{E}\eta = 0, \textbf{E}\eta^2 =1, \textbf{E}\eta^s= O \left(\frac{1}{(\sqrt{p})^ {s-2}} \right)~\text{for}~ s\ge 2.$$

\medskip
For a positive integer $k$, the $k^{\text{th}}$ moment of ESD of the matrix $W_n$ is 
$$\displaystyle\int x^k dF_n^{W}(x)= \frac{1}{n} \textbf{E}( \text{Trace}({W_n}^k)),$$

and the $k^{\text{th}}$ moment of the semicircle distribution is 
$$\displaystyle\int_{-2}^{2} x^k \rho_{\text{sc}}(x) dx.$$

On a compact set, convergence in distribution is the same as convergence of moments. To prove the theorem, we need to show, for every fixed number $k$,

\begin{equation}\label{conver}
\frac{1}{n} \textbf{E}( \text{Trace}({W_n}^k)) \rightarrow \displaystyle\int_{-2}^{2} x^k \rho_{\text{sc}}(x) dx, \ \text{as}~ n \rightarrow \infty.
\end{equation}

For $k=2m+1$, by symmetry, $\displaystyle\int_{-2}^{2} x^k \rho_{\text{sc}}(x) dx=0$. 

For $k=2m$, 
\begin{equation*}
\begin{split}
\displaystyle\int_{-2}^{2} x^k \rho_{\text{sc}}(x) dx 
&= \frac{1}{\pi}\int_{0}^{2} x^k \sqrt{4-x^2} dx
= \frac{2^{k+2}}{\pi}\int_{0}^{\pi/2} {\sin^k{\theta}} {\cos^2{\theta}} dx\\
&= \frac{2^{k+2}}{\pi} \frac{\Gamma(\frac{k+1}{2})\Gamma(\frac{3}{2})}{\Gamma(\frac{k+4}{2})}
=\frac{1}{m+1} \dbinom{2m}{m}
\end{split}
\end{equation*}

Thus our claim (\ref{conver}) follows by showing that 

\begin{equation}\label{trace}
\displaystyle\frac{1}{n} \textbf{E}( \text{Trace}({W_n}^k)) =\left\{ \begin{array}{ll}
         O(\frac{1}{\sqrt{np}}) & \mbox{if $k = 2m+1$};\\
         \\
         \frac{1}{m+1} {{2m}\choose{m}} + O(\frac{1}{np}) & \mbox{if $k = 2m$}.\end{array} \right.
\end{equation}

We have the expansion for the trace of ${W_n}^k$, 
\begin{equation}\label{expan}
\begin{split}
\displaystyle\frac{1}{n} \textbf{E}( \text{Trace}({W_n}^k))
& =\frac{1}{n^{1+k/2}} \textbf{E}( \text{Trace}({\sigma}^{-1} M_n)^k )\\
& =\frac{1}{n^{1+k/2}} \sum_{1 \le i_1, \ldots, i_k \le n} \textbf{E} \eta_{i_1 i_2} \eta_{i_2 i_3} \cdots \eta_{i_k i_1}
\end{split}
\end{equation}

Each term in the above sum corresponds to a closed walk of length $k$ on the complete graph $K_n$ on $\{1,2, \ldots, n \}$. On the other hand, $\eta_{ij}$ are independent with mean 0. Thus the term is nonzero if and only if every edge in this closed walk appears at least twice. And we call such a walk a \emph{good} walk. Consider a \emph{good} walk that uses $l$ different edges $e_1, \ldots, e_l$ with corresponding multiplicities $m_1, \ldots, m_l$, where $l \le m$, each $m_h \ge 2$ and $m_1+\ldots+m_l = k$. Now the corresponding term to this \emph{good} walk has form $$\textbf{E} \eta_{e_1}^{m_1}\cdots \eta_{e_l}^{m_l}.$$

\medskip
Since such a walk uses at most $l+1$ vertices, a naive upper bound for the number of \emph{good} walks of this type is $n^{l+1} \times l^k$. 

\medskip
When $k=2m+1$, recall $\textbf{E}\eta^s= \Theta \left({(\sqrt{p})^ {2-s}} \right)~\text{for}~ s\ge 2$, and so

\begin{equation*}
\begin{split}
\displaystyle\frac{1}{n} \textbf{E}( \text{Trace}({W_n}^k))
& = \frac{1}{n^{1+k/2}} \sum_{l=1}^{m} \sum_{\text{\emph{good} walk of l edges}} \textbf{E} \eta_{e_1}^{m_1}\cdots \eta_{e_l}^{m_l}\\
& \le \frac{1}{n^{m+3/2}} \sum_{l=1}^{m} n^{l+1} l^k (\frac{1}{ \sqrt{p}})^{m_1-2}\ldots (\frac{1}{ \sqrt{p}})^{m_l-2}\\
& = O(\frac{1}{\sqrt{np}}). 
\end{split}
\end{equation*}

When $k=2m$, we classify the \emph{good} walks into two types. The first kind uses $l \le m-1$ different edges. The contribution of these terms will be 

\begin{equation*}
\begin{split}
\frac{1}{n^{1+k/2}} \sum_{l=1}^{m-1} \sum_{\text{1st kind of \emph{good} walk of l edges}} \textbf{E} \eta_{e_1}^{m_1}\cdots \eta_{e_l}^{m_l}
& \le \frac{1}{n^{1+m}} \sum_{l=1}^{m} n^{l+1} l^k (\frac{1}{ \sqrt{p}})^{m_1-2}\ldots (\frac{1}{ \sqrt{p}})^{m_l-2}\\
& = O(\frac{1}{{np}}). 
\end{split}
\end{equation*}

The second kind of \emph{good} walk uses exactly $l=m$ different edges and thus $m+1$ different vertices. And the corresponding term for each walk has form $$\textbf{E} \eta_{e_1}^{2}\cdots \eta_{e_l}^{2}=1.$$

The number of this kind of \emph{good} walk is given by the following result in the paper (\cite{bai08method}, Page 617--618):

\begin{lem}
The number of the second kind of \emph{good} walk is $$\displaystyle\frac{n^{m+1}(1+O(n^{-1}))}{m+1} \dbinom{2m}{m}.$$
\end{lem}

Then the second conclusion of (\ref{conver}) follows. 


\section{ Proof of Lemma \ref{ConcenLem}:}\label{appendix:projection}

The coordinates of $Y$ are bounded in magnitude by $1$. Apply Talagrand's inequality to the map $Y \rightarrow ||\pi_H(Y)||$, which is convex and $1$-Lipschitz. We can conclude 

\begin{equation} \label{eq:1.2}
\textbf{P}(|\parallel \pi_H (Y) \parallel -M(\parallel \pi_H (Y) \parallel)| \ge t) \le 4 \exp(-\frac{t^2}{16})
\end{equation} 
where $M(\parallel \pi_H (Y) \parallel)$ is the median of $\parallel \pi_H (Y) \parallel$.

\medskip
Let $P=(p_{ij})_{1 \le i,j \le n}$ be the orthogonal projection matrix onto $H$. One has trace$P^2=$trace$P= \sum_i p_{ii}=d$ and $|p_{ii}| \le 1$, as well as,

$${\parallel \pi_H (Y) \parallel}^2 = \sum_{1 \le i,j \le n}^{} p_{ij} \zeta_i \zeta_j = \sum_{i=1}^{n} p_{ii} \zeta_{i}^2 + \sum_{i\neq j} p_{ij} \zeta_{i} \zeta_j$$

and 

$$\mathbf{E}{\parallel \pi_H (Y) \parallel}^2=\mathbf{E}(\sum_{i=1}^{n} p_{ii} \zeta_{i}^2)+\mathbf{E}(\sum_{i\neq j} p_{ij} \zeta_{i} \zeta_j)= \sigma^2 d.$$
\medskip
Take $L=4/\sigma$. To complete the proof, it suffices to show 
\begin{equation} \label{eq:1.3}
|M(\parallel \pi_H (Y) \parallel) - \sigma \sqrt{d}| \le L\sigma.
\end{equation}

Consider the event $\mathcal{E}_{+}$ that $\parallel \pi_H (Y) \parallel \ge \sigma L + \sigma \sqrt{d}$, which implies that ${\parallel \pi_H (Y) \parallel}^2 \ge \sigma^2(L^2 + 2L\sqrt{d} +d^2).$

\medskip
Let $S_1= \sum_{i=1}^{n}p_{ii} (\zeta_{i}^2-\sigma^2)$ and $S_2=\sum_{i \neq j}^{} p_{ij}\zeta_i \zeta_j$.

\medskip
Now we have 

$$\textbf{P} (\mathcal{E}_{+}) \le \textbf{P} (\sum_{i=1}^{n} p_{ii}\zeta_{i}^2 \ge \sigma^2 d + L\sqrt{d}\sigma^2) + \textbf{P} (\sum_{i \neq j}^{} p_{ij} \zeta_i \zeta_j \ge \sigma^2 L\sqrt{d}).$$

By Chebyshev's inequality, 

$$\textbf{P} (\sum_{i=1}^{n} p_{ii}\zeta_{i}^2 \ge \sigma^2 d + L\sqrt{d}\sigma^2) =\textbf{P} (S_1 \ge  L\sqrt{d}\sigma^2)) \le \frac{\textbf{E}(|S_1|^2)}{L^2 d \sigma^4},$$

where $\textbf{E}(|S_1|^2) = \textbf{E} (\sum_{i} p_{ii} (\zeta_{i}^2 -\sigma^2) )^2 = \sum_{i} p_{ii}^2 \textbf{E} (\zeta_i^4- \sigma^4) \le d\sigma^2(1-2\sigma^2)$.

\medskip
Therefore, $\textbf{P}(S_1 \ge L\sqrt{d}\sigma^4) \le \displaystyle\frac{d\sigma^2(1-2\sigma^2)}{L^2 d \sigma^4} < \frac{1}{16}.$

\medskip
On the other hand, we have $\textbf{E}(|S_2|^2)=\textbf{E}(\sum_{i\neq j}^{} p_{ij}^2 \zeta_i^2 \zeta_j^2) \le \sigma^4 d$ and 

$$\textbf{P} (\sum_{i \neq j}^{} p_{ij} \zeta_i \zeta_j \ge \sigma^2 L\sqrt{d}) = \textbf{P} (S_2 \ge L\sqrt{d} \sigma^2) \le \frac{\textbf{E}(|S_2|^2)}{L^2 d \sigma^4} < \frac{1}{10}$$

\medskip
It follows that $\textbf{E}(\mathcal{E}_{+}) < 1/4$ and hence $M(\parallel \pi_H (Y) \parallel) \le L\sigma +\sqrt{d} \sigma.$ 

\medskip
For the lower bound, consider the event $\mathcal{E}_{-} $ that $\parallel \pi_H (Y) \parallel \le \sqrt{d}\sigma-L\sigma$ and notice that 

$$\textbf{P}(\mathcal{E}_{-}) \le \textbf{P} (S_1 \le - L\sqrt{d} \sigma^2) + \textbf{P} (S_2 \le - L\sqrt{d} \sigma^2).$$

\medskip
The same argument applies to get $M(\parallel \pi_H (Y) \parallel) \ge \sqrt{d}\sigma-L\sigma$. Now the relations (\ref{eq:1.2}) and (\ref{eq:1.3}) together imply (\ref{eq:1.1}).

\endproof


\section{Proof of Lemma \ref{ConBulkAdj}: } \label{appendix:concentration}

Recall the normalized adjacency matrix
$$M_n=\frac{1}{\sigma}(A_n-p J_n),$$
where $J_n=\bold{1}_n \bold{1}^T_n $ is the $n \times n$ matrix of all $1$'s, and let  $W_n=\frac{1}{\sqrt{n}}M_n$.

\begin{lem} 
\label{Lem1}
For all intervals $I \subset \mathbb{R}$ with $|I| = \omega{(\log n)}/{np}$, one has 
$$N_I(W_n)=O(n|I|)$$ with overwhelming probability.

\end{lem}

The proof of Lemma \ref{Lem1} uses the same proof as in the paper \cite{tvrandom} with the relation (\ref{eq:1.1}).

Actually we will prove the following concentration theorem for $M_n$. By Lemma \ref{EigenDiff}, $| N_I (W_n)-N_I(B_n) | \le 1$, therefore   Lemma \ref{ConBulk} implies Lemma \ref{ConBulkAdj}.

\begin{lem}
\label{ConBulk}
(Concentration for ESD in the bulk) Assume $p={g(n)\log n}/{n}$. For any constants $\varepsilon, \delta > 0$ and any interval $I$ in $[-2+\varepsilon, 2-\varepsilon]$ of width $|I|=\Omega( g(n)^{0.6}\log n/{np} )$, the number of eigenvalues $N_I$ of $W_n=\frac{1}{\sqrt{n}} M_n$ in $I$ obeys the concentration estimate $$|N_I(W_n) - n \displaystyle\int_I {{}\rho}_{sc}(x)\,dx| \le {\delta} n |I|$$ with overwhelming probability. 
\end{lem}

To prove Theorem \ref{ConBulk}, following the proof in \cite{tvrandom}, we consider the \textit{Stieltjes transform} 
$$s_n(z):=\frac{1}{n} \displaystyle\sum_{i=1}^{n} \frac{1}{\lambda_i(W_n)-z},$$ whose imaginary part
$$\text{Im} s_n(x+\sqrt{-1} \eta)=\frac{1}{n} \displaystyle\sum_{i=1}^{n} \frac{\eta}{\eta^2 + (\lambda_i(W_n)-x)^2}>0$$in the upper half-plane $\eta >0$.

\medskip
The semicircle counterpart
$$s(z):= \displaystyle\int_{-2}^{2} \frac{1}{x-z} \rho_{sc}(x)\,dx=\frac{1}{2\pi}\displaystyle\int_{-2}^{2} \frac{1}{x-z} \sqrt{4-x^2}\,dx,$$ is the unique solution to the equation $$s(z)+\frac{1}{s(z)+z}=0$$ with $\text{Im} s(z) >0$. 

\medskip
The next proposition gives control of ESD through control of Stieltjes transform (we will take $L=2$ in the proof):

\begin{prop}
\emph{(Lemma 60, \cite{tvrandom})}  
\label{ESDStie}
Let $L, \varepsilon, \delta >0$. Suppose that one has the bound $$|s_n(z)-s(z)| \le \delta$$ with (uniformly) overwhelming probability for all $z$ with $|\text{Re}(z)| \le L$ and $\text{Im}(z) \ge \eta$. Then for any interval $I$ in $[-L+\varepsilon, L-\varepsilon]$ with $|I| \ge \text{max}(2\eta, \frac{\eta}{\delta} \log \frac{1}{\delta})$, one has $$|N_I- n \displaystyle\int_I {\rho}_{sc}(x)\,dx| \le \delta n |I|$$ with overwhelming probability. 
\end{prop}

By Proposition \ref{ESDStie}, our objective is to show 

\begin{equation}
|s_n(z)-{s}(z)| \le {\delta}
\end{equation}
with (uniformly) overwhelming probability for all $z$ with $|\text{Re}(z)| \le 2$ and $\text{Im}(z) \ge {\eta}$, where $$\eta =\frac{\log^2 g(n) \log n}{np}.$$

In Lemma \ref{StieTran}, we write 

\begin{equation} \label{eq:1.6}
s_n(z)= \frac{1}{n} \displaystyle{\sum_{k=1}^{n} \frac{1}{-\frac{{\zeta}_{kk}}{ \sqrt{n}\sigma}-z- Y_k }}
\end{equation}
where $$Y_k=a^*_k (W_{n,k} -zI)^{-1} a_k,$$ 
$W_{n,k}$ is the matrix $W_n$ with the $k^{\text{th}}$ row and column removed, and $a_k$ is the $k^{\text{th}}$ row of $W_n$ with the $k^{\text{th}}$ element removed. 

\medskip
The entries of $a_k$ are independent of each other and of $W_{n,k}$, and have mean zero and variance $1/n$. By linearity of expectation we have $$\mathbf{E}(Y_k|W_{n,k})=\frac{1}{n}\text{Trace}(W_{n,k}-zI)^{-1}=(1-\frac{1}{n})s_{n,k}(z)$$ where $$s_{n,k}(z)= \frac{1}{n-1} \displaystyle{\sum_{i=1}^{n-1} \frac{1}{\lambda_i (W_{n,k}) -z}}$$ is the \textit{Stieltjes transform} of $W_{n,k}$. From the Cauchy interlacing law, we get$$\displaystyle{|{} s_n(z)- (1-\frac{1}{n}) {} s_{n,k}(z)|= O(\frac{1}{n} \int_{\mathbb{R}} \frac{1}{|x-z|^2}\,dx) =O(\frac{1}{n\eta})}=o(1),$$ 
and thus $$\mathbf{E}(Y_k|W_{n,k})=s_n(z)+o(1).$$

In fact a similar estimate holds for $Y_k$ itself:
\begin{prop}
\label{YProp}
For $1 \le k \le n$, $Y_k= \mathbf{E}(Y_k|W_{n,k}) +o(1)$ holds with (uniformly) overwhelming probability for all $z$ with $|\text{Re}(z)| \le 2$ and $\text{Im}(z) \ge {\eta}$.
\end{prop}

Assume this proposition for the moment. By hypothesis, $|\frac{{\zeta}_{kk}}{ \sqrt{n}\sigma}|=|\frac{-p}{\sqrt{n}\sigma}|=o(1)$. Thus in (\ref{eq:1.6}), we actually get

\begin{equation}
{} s_n(z) + \frac{1}{n} \displaystyle{\sum_{k=1}^{n} \frac{1}{s_n(z) +z + o(1)}}=0
\end{equation}

with overwhelming probability. This implies that with overwhelming probability either $s_n(z)=s(z)+o(1)$ or that $s_n(z)=-z+o(1)$. On the other hand, as Im$s_n(z)$ is necessarily positive, the second possibility can only occur when Im$z=o(1)$. A continuity argument (as in \cite{erdos09local}) then shows that the second possibility cannot occur at all and the claim follows.

\medskip

Now it remains to prove Proposition \ref{YProp}.

\medskip
{\bf Proof of Proposition \ref{YProp}.} Decompose $$Y_k=\displaystyle{\sum_{j=1}^{n-1} \frac{|u_j (W_{n,k})^*a_k|^2}{\lambda_{j}(W_{n,k})-z}}$$

and evaluate
\begin{equation}
\begin{split}
Y_k- \mathbf{E}(Y_k|W_{n,k})&= Y_k- \displaystyle{(1-\frac{1}{n}) {} s_{n,k}(z)}+o(1)\\
&= \displaystyle{\sum_{j=1}^{n-1} \frac{|u_j (W_{n,k})^*a_k|^2- \frac{1}{n}}{\lambda_{j}(W_{n,k})-z}}+o(1)\\
&=  \displaystyle{\sum_{j=1}^{n-1} \frac{R_j}{\lambda_{j}(W_{n,k})-z}}+o(1),
\end{split}
\end{equation}

where we denote $R_j=\displaystyle |u_j (W_{n,k})^*a_k|^2- \frac{1}{n}$, $\{u_j (W_{n,k})\}$ are orthonormal eigenvectors of $W_{n,k}$.

\medskip
Let $J \subset \{1,\ldots, n-1\}$, then  

$$\displaystyle\sum_{j\in J} R_j =||P_H(a_k)||^2- \frac{\text{dim}(H)}{n}$$

where $H$ is the space spanned by $\{u_j (W_{n,k})\}$ for $j \in J$ and $P_H$ is the orthogonal projection onto $H$.

In Lemma \ref{ConcenLem}, by taking $t=h(n)\sqrt{\log n}$, where $h(n)=\log^{0.001} g(n) $, one can conclude with overwhelming probability 

\begin{equation} \label{eq:1.9}
|\displaystyle\sum_{j\in J} R_j| \ll \frac{1}{n}\left(\frac{h(n)\sqrt{|J|\log n}}{\sqrt{p}}+\frac{h(n)^2 \log n}{p}\right).
\end{equation}

Using the triangle inequality, 

\begin{equation} \label{eq:1.10}
\displaystyle\sum_{j\in J} |R_j| \ll \frac{1}{n}\left(|J| +\frac{h(n)^2 \log n}{p} \right)
\end{equation}

with overwhelming probability.

\medskip
Let $z=x+\sqrt{-1} {} \eta$, where $\eta =\log^2 g(n) \log n/{np}$ and $|x| \le 2 -\varepsilon$, define two parameters 
$$\alpha = \frac{1}{\log^{4/3} g(n)} ~~~~~~~\text{and}~~~~~~~ \beta=\frac{1}{\log^{1/3} g(n) }.$$

\medskip
First, for those $j \in J$ such that $|\lambda_j(W_{n,k})-x| \le \beta \eta $, the function $\frac{1}{\lambda_j(W_{n,k})-x -\sqrt{-1}{} \eta}$ has magnitude $O(\frac{1}{{} \eta})$. From Lemma \ref{Lem1}, $|J| \ll n\beta \eta$, and so the contribution for these $j \in J$ is,

$$\displaystyle{|\sum_{j \in J}^{} \frac{R_j}{\lambda_{j}(W_{n,k})-z}|} \ll \frac{1}{n\eta} \left( n\beta \eta +\frac{h(n)^2 }{\log^2 g(n)} \right) = O(\frac{1}{ \log^{1/3}g(n) })=o(1).$$

\medskip
For the contribution of the remaining $j\in J$, we subdivide the indices as 

$$a \le |\lambda_j(W_{n,k})-x| \le (1+\alpha)a$$

where $a=(1+\alpha)^l \beta \eta$, for $0 \le l \le L$, and then sum over $l$.

\medskip
For each such interval, the function $\frac{1}{\lambda_j(W_{n,k})-x -\sqrt{-1}{} \eta}$ has magnitude $O(\frac{1}{a})$ and fluctuates by at most $O(\frac{\alpha}{a})$. Say $J$ is the set of all $j$'s in this interval, thus by Lemma \ref{Lem1}, $|J| =O(  n\alpha a)$. Together with bounds (\ref{eq:1.9}), (\ref{eq:1.10}), the contribution for these $j$ on such an interval,

\begin{equation*}
\begin{split}
\displaystyle{|\sum_{j \in J}^{} \frac{R_j}{\lambda_{j}(W_{n,k})-z}|} &\ll \frac{1}{an} \left(\frac{h(n)\sqrt{|J|\log n}}{\sqrt{p}}+\frac{h(n)^2 \log n}{p} \right)+  \frac{\alpha}{an}  \left(|J| +\frac{h(n)^2 \log n}{p} \right)\\
&=O\left(  \frac{ \sqrt{\alpha} }{ \sqrt{(1+\alpha)^l} } \frac{ h(n) }{ \sqrt{\beta}\log g(n) }   + \frac{h^2(n)}{ (1+\alpha)^l \beta \log^2 g(n) } +\alpha^2    \right)\\
&= O\left( \frac{1}{ \sqrt{\alpha\beta} } \frac{h(n)}{ \log g(n) }  + \alpha \log \frac{1}{\beta \eta} \right)
\end{split}
\end{equation*}

\medskip
Summing over $l$ and noticing that $(1+\alpha)^{L} \eta/g(n)^{1/4} \le 3$, we get 
\begin{equation*}
\begin{split}
\displaystyle{|\sum_{j \in J, \text{all} J}^{} \frac{R_j}{\lambda_{j}(W_{n,k})-z}|} &= O\left( \frac{1}{\sqrt{\alpha \beta}} \frac{h(n)}{\log g(n)} + \alpha \log\frac{1}{\beta \eta} \right)\\
&=O\left(  \frac{h(n)}{\log^{1/6} g(n)} \right) =o(1).
\end{split}
\end{equation*}
\hfill $\Box$

\end{appendices}

{\bf\Large Acknowledgement.} The authors thank Terence Tao for useful conversations.


\bibliographystyle{plain}
\bibliography{sparse-bulk}

\end{document}